\documentclass[english,12pt]{article}
\usepackage[latin1]{inputenc}
\usepackage[dvips]{graphics}
\usepackage{amsmath,amsthm,amssymb,babel}
\usepackage{color}
\usepackage{graphicx}
\usepackage{ulem}
\usepackage{siunitx}

\newtheorem{remark}{Remark}

\newcommand{\R}{{\if mm {\rm I}\mkern -3mu{\rm R}\else \leavevmode
\hbox{I}\kern -.17em\hbox{R} \fi}}

\newcommand{\bu}{\mbox{\boldmath{$u$}}}

\newcommand{\bv}{\mbox{\boldmath{$v$}}}
\newcommand{\bw}{\mbox{\boldmath{$w$}}}
\newcommand{\bo}{\mbox{\boldmath{$\omega$}}}

\newcommand{\bx}{\mbox{\boldmath{$x$}}}
\newcommand{\bxi}{\mbox{\boldmath{$\xi$}}}

\newcommand{\bz}{\mbox{\boldmath{$z$}}}

\newcommand{\fb}{\mbox{\boldmath{$f$}}}

\newcommand{\bsigma}{\mbox{\boldmath{$\sigma$}}}
\newcommand{\btau}{\mbox{\boldmath{$\tau$}}}

\newcommand{\bvarepsilon}{\mbox{\boldmath{$\varepsilon$}}}
\newcommand{\bnu}{\mbox{\boldmath{$\nu$}}}

\newcommand{\bzero}{\mbox{\boldmath{$0$}}}

\newcommand{\cS}{\mbox{{${\cal S}$}}}
\newcommand{\cR}{\mbox{{${\mathcal R}$}}}

\newcommand{\cP}{\mbox{{${\cal P}$}}}

\newcommand{\cA}{\mbox{{${\cal A}$}}}
\newcommand{\cB}{\mbox{{${\cal B}$}}}

\def\real{\mathbb{R}}

\newtheorem{theorem}{Theorem}[section]
\newtheorem{lemma}[theorem]{Lemma}

\newtheorem{definition}[theorem]{Definition}
\newtheorem{proposition}[theorem]{Proposition}

\numberwithin{equation}{section}

\title{\bf Duality Arguments in the Analysis of a~Viscoelastic  Contact  Problem}

	\vspace{5mm}

	\vspace{22mm}
	{\author{Piotr Bartman$^{1,2}$,\ Anna Ochal$^{2}$,\ Mircea Sofonea$^{3}$\\[4mm]
                {\it \small $1$ Doctoral School of Exact and Natural Sciences, Jagiellonian University}
                \\{\it\small Lojasiewicza 11, 30348 Krakow, Poland}
                \\ [4mm]
                {\it \small $2$ Chair of Optimization and Control,}
			{\it \small Jagiellonian University}
			\\{\it\small Lojasiewicza 6, 30348 Krakow, Poland}		\\[4mm]
			{\it \small $3$ Laboratoire de Math\'ematiques et Physique}
	        {\it \small
				University of Perpignan Via Domitia}
			\\{\it\small 52 Avenue Paul Alduy, 66860 Perpignan, France}		}}

\begin{document}
\maketitle

\vskip4mm
\begin{abstract}
\noindent We consider a  mathematical model  which describes the quasistatic frictionless contact of a viscoelastic body with a rigid-plastic foundation.  We describe the mechanical assumptions,
list the hypotheses on the data and provide three different variational formulations of the model in which the unknowns are the displacement field, the stress field and the strain field, respectively. These  formulations have a different structure. Nevertheless,  we prove that they are pairwise dual of each other.
Then, we deduce the unique weak solvability of the
contact problem as well as the Lipschitz continuity of its weak solution  with respect to the data.
The proofs are based on recent results on history-dependent variational inequalities and  inclusions.
Finally, we present numerical simulations in the study of the contact problem, together with the corresponding mechanical interpretations.

\end{abstract}

\bigskip \noindent {\bf AMS Subject Classification\,:}   74M15, 74M10,
47J22,  49J40, 49J21, 34G25.

\bigskip \noindent {\bf Key words\,:}
viscoelastic material, frictionless  contact, history-dependent variational inequality, history-dependent inclusion, weak solution, numerical simulations.

Dedicated to Professor Zhenhai Liu on the occasion of his 65th birthday.

\vskip15mm

\section{Introduction}\label{int}

Contact phenomena between deformable bodies arise in industry and everyday life. They are modeled by strongly nonlinear boundary value problems  which usually do not have classical solutions. Therefore, their study is made by using a variational approach, that consists to replace the strong formulation of the problem by
a weak or variational formulation, which is more convenient for mathematical analysis and numerical simulations.

The weak formulations of contact problems vary from problem to problem, from author to author and even from paper to paper.
{They lead to challenging nonlinear problems which, in general, are expressed in terms of either variational and hemivariational inequalities or inclusions, including differential inclusions. Comprehensive references in the theory of variational inequalities are \cite{BC,B} and, more recently, \cite{GJKR}. There, various existence and uniqueness results are presented, obtained by using different functional arguments. Hemivariational inequalities are inequality problems governed by a locally Lipschitz continuous function. Their analysis is carried out by using arguments of pseudomonotonicity for multivalued operators combined with the properties of the generalized directional derivative and the subdiffrential in the sense of Clarke. Basic references in the field are \cite{NP, P}.
Finally, for the theory of differential inclusion,we mention the book \cite{KOZ} and the survey paper \cite{ZOZ}. The book \cite{KOZ} deals with the theory of semilinear differential inclusions in infinite dimensional spaces, in a setting in which neither convexity of the map or compactness of the multi-operators is supposed. There, arguments of degree theory are used for solving operator inclusions, fixed points and optimization problems. The theory is applied to the investigation of semilinear differential inclusions in Banach spaces. In the survey paper \cite{ZOZ}
the authors discuss applications of differential and operator inclusions to some optimization and optimal control problems, including an optimal feedback control problem for a mathematical model of the motion of weakly concentrated water polymer solutions. }

For most of the problems which describe the contact of a viscoelastic material, the variational formulation is given in a form of a variational inequality with time-independent unilateral constraints in which the unknown is the displacement field.
References on this topic include  \cite{C,DL,EJK,HS,KO, P1,SM}.  Nevertheless, for several problems it is more convenient to consider the stress field as the main unknown and, therefore, to obtain a variational formulation
in term of the stress field. Such a formulation is usually in a form of  a variational inequality too, but it has a different structure since in this case the unilateral constraints are time-dependent. References in the field are \cite{MOS,S,SM}, for instance. Besides the displacement and the stress fields, the strain field can be successfully used to study various contact problems, as proved recently. Choosing the strain field as the main unknown leads to a variational formulation which is in the form of a history-dependent inclusion or a sweeping process.
Reference in the field are  \cite{AH,AS,NS2,SX}, for instance.

The aim of this current paper is two fold.
The first one is to provide three different variational formulations for a viscoelastic contact problem (in which the unknowns are the displacement, the stress and the strain field, respectively), to prove their equivalence and their unique solvability, as well.  Our proofs show that the corresponding variational formulations are pairwise dual to each other (in the sense introduced in \cite{RS1,S}), which consists the first trait of novelty of our work.
Our second aim in this paper is to introduce a numerical approximation scheme of the problem (based on the variational formulation in displacements) and to provide numerical simulations together with the corresponding mechanical interpretations.
This represents the second novelty of the current  paper.

The rest of the manuscript is organized as follows. In Section \ref{s2} we present some notation and preliminary material which are  needed in the next sections. This concerns the properties of the function spaces
we use, a result on the history-dependent operators and some abstract results for history-dependent variational inequalities and inclusions. In Section \ref{s3} we introduce the viscoelastic model of contact and we provide a description of the equations and boundary value conditions.
Then, we list the hypotheses on the data.
In Section \ref{s4} we consider there variational formulations of the problem and prove that these formulations are pairwise dual of each other.
Then, in Section \ref{s5}  we state and prove existence and uniqueness results, which allow us to define the concept of a weak solution to the contact model.
Finally, we  end this paper with Section \ref{s7} in which we present a numerical scheme for the displacement variational formulation, together with some numerical simulations and the corresponding mechanical interpretations.

\section{Notation and preliminaries}\label{s2}
\setcounter{equation}0

The preliminary material we present in this section concerns  basic notation, an existence and uniqueness result for a class of  time-dependent inclusions, and some properties of the function spaces in Contact Mechanics. Everywhere in this section $X$ represents a real Hilbert space endowed with an inner product $(\cdot,\cdot)_X$ and its associated norm $\|\cdot\|_X$, and  $2^{X}$ denotes the set of parts of $X$.

\medskip\noindent
{\bf Basic notation.}  We use the notation $N_K$ for
the outward normal cone of a  nonempty closed convex subset $K\subset X$.
It is well known that
 $N_K \colon X\to 2^X$  and, for any $u,\,f\in X$, we have
\begin{eqnarray}
&&\label{5}
f\in N_K(u)\ \Longleftrightarrow\  u\in K,\quad (f,v-u)_X\le 0\quad\text{for all}\ \, v\in K.
\end{eqnarray}
We also recall that a convex function $\varphi\colon X\to\mathbb{R}$ is said to be subdifferentiable (in the
sense of the convex analysis) if for any $u\in X$ there exists an element $\xi\in X$ such that
\begin{equation*}\label{sub}
\varphi(v)-\varphi(u)\ge (\xi,v - u)_X
\ \ \ \ \mbox{for all} \ \ \,v \in X.
\end{equation*}

Consider now an interval of time $[0,T]$ with $T>0$. We denote by $C([0,T];X)$  the space of continuous functions defined on $[0,T]$ with values in $X$. Then, it is well known that
$C([0,T];X)$ is a Banach space
equipped with the norm
\begin{equation}\label{cann} \|v\|_{C([0,T];X)} =
	\max_{t\in [0,T]}\,\|v(t)\|_X.
\end{equation}
For an operator $\cS \colon C([0,T];X)\to C([0,T];X)$ and a function $u\in C([0,T];X)$   we use the shorthand notation $\cS u(t)$ to
represent the value of the function $\cS u$ at the point $t\in [0, T]$, that is, $\cS u(t):=(\cS u)(t)$. Moreover, if $A\colon X\to X$, then
$A+\cS$ will represent a shorthand notation for the operator  which maps any function $u\in C([0,T];X)$ to the function
$t\mapsto Au(t)+\cS u(t)\in C([0,T];X)$.

\begin{definition}\label{d0}
	An operator $\cS\colon C([0,T];X)$ $\to C([0,T];X)$
	is said to be a history-dependent operator if there exists  $L>0$ such that
	\begin{eqnarray}
		&&\label{Ahd}\|\cS u(t)-\cS v(t)\|_X\le L\displaystyle\int_0^t
		\|u(s)-v(s)\|_X\,ds\ \quad\forall\,u,\, v\in C([0,T];X),\ t\in [0,T].\nonumber
	\end{eqnarray}
\end{definition}

History-dependent operators arise in Functional Analysis, Solid Mechanics and Contact Mechanics, as well.
General properties, examples and mechanical interpretations can be found in \cite{SM}.
An important  property of history-dependent operators which will be useful in this paper is the following.

\begin{theorem}\label{t0} Let $\widetilde{A}\colon X\to {X}$ be a linear continuous operator such that
\begin{equation*}\label{AAA}
 (\widetilde{A}u,u)_X\ge m \|u\|_X^2\ \ \forall\, u\in X
\end{equation*}
with some $m>0$ and consider a history-dependent ope\-ra\-tor
$\widetilde{\cS}\colon C([0,T];X)\to C([0,T];X)$. Then the operator $\widetilde{A}+\widetilde{\cS}\colon C([0,T];X)\to C([0,T];X)$ is invertible and its inverse is of the form $\widetilde{A}^{-1}+\widetilde{\cR}\colon C([0,T];X)\to C([0,T];X)$, where  $\widetilde{A}^{-1}:X\to X$ represents the inverse of the operator $\widetilde{A}$ and  $\widetilde{\cR}\colon C([0,T];X)\to C([0,T];X)$ is a~history-dependent operator.
\end{theorem}

A proof of Theorem \ref{t0} can be found in \cite[p. 55]{SMBOOK}, based on results on nonlinear implicit equations in Banach spaces.

\medskip\noindent
{\bf  History-dependent variational inequalities and inclusions.}  Consider a~set~$K$, the operators
$A$,  $\mathcal{S}$, a function $f$ and a  set-valued  mapping $\Sigma$,
which satisfy the following conditions.

\begin{itemize}

	\item [$(\mathcal{K})$\ \ ]  $K\subset X$ is a nonempty closed convex subset.

	\item [$(\mathcal{A})$\ \ ]  $A:X\to X$ is a strongly monotone and Lipschitz continuous operator.

\item [$(\mathcal{S})$\ \ ] $\cS\colon C([0,T];X)\to C([0,T];X)$ is a history-dependent operator.

\item [$(j)$\ \ ] $j\colon X\to\mathbb{R}$ is a convex lower semicontinuous function.

\item [$({f})$\ \ ]$f\in C([0,T];X)$.

\item [$(\Sigma)$\ \ ] $\Sigma\colon [0,T]\rightarrow 2^X$ and  there exist
a nonempty closed convex set $\Sigma_0\subset X$ and   a~function  $g\in C([0,T];X)$  such that $\Sigma(t)= \Sigma_0+g(t)$ for all $t\in[0,T]$.

\end{itemize}

We have the following existence and uniqueness results.

\begin{theorem}\label{t1}	Assume $(\mathcal{K})$, $(\mathcal{A})$, $(\mathcal{S})$, $(j)$ and $(f)$.
Then, there exists a unique function $u\in C([0,T];X)$ such that for all $t\in[0,T]$ the following inequality holds:
\begin{eqnarray}\label{zz}
&u(t)\in K,&(Au(t),v-u(t))_X+(\cS u(t),v-u(t))_X\\ [2mm]
&&\qquad+j(v)-j(u(t))\ge (f(t),v-u)_X\qquad\forall\, v\in K.\nonumber
\end{eqnarray}

\end{theorem}

\begin{theorem}\label{t2}	Assume $(\mathcal{A})$, $(\mathcal{S})$ and $(\Sigma)$.
	Then, there exists a unique function $u\in C([0,T];X)$ such that for all $t\in[0,T]$ the following inclusion holds:
	\begin{equation*}
	-u(t)\in N_{\Sigma(t)}\big(Au(t)+\cS u(t)\big).
	\end{equation*}

\end{theorem}

Theorem \ref{t1}   represents a direct consequence of a result proved in \cite[Ch.3]{SM} while
Theorem \ref{t2}  is a direct consequence of a result proved in \cite[Ch.6]{S}.
Their proofs are based on arguments of convex analysis, monotone operators and a fixed point result for  history-dependent operators.

\medskip\noindent
{\bf Function spaces.} For the contact problem we consider in this paper we introduce some specific notation we shall need in the following sections. First, $\mathbb{S}^d$  stands for the space of second order symmetric
tensors on $\mathbb{R}^d$ with $d\in\{2,3\}$. Moreover, $``\cdot"$ and $\| \cdot\|$  represent the inner product and the Euclidean norm on the spaces  $\mathbb{R}^d$  and  $\mathbb{S}^d$, respectively.
In addition,
$\Omega\subset\mathbb{R}^d$  is a
bounded domain with a
Lipschitz continuous boundary $\Gamma$.
The outward unit normal at $\Gamma$ will be denoted by $\bnu$, and $\Gamma_1$ is a~measurable part of $\Gamma$ with positive measure.

We use the standard notation for the Lebesgue and Sobolev spaces associated to $\Omega$ and $\Gamma$. Typical examples are the spaces
$L^2(\Omega)^d$, $L^2(\Gamma)^d$ and $H^1(\Omega)^d$ equipped with their canonical Hilbertian structure.
For an element $\bv\in H^1(\Omega)^d$ we still write $\bv$ for the trace $\gamma\bv\in L^2(\Gamma)^d$ and $v_\nu$, $\bv_\tau$ for the normal and
tangential traces on the boundary, i.e.,
$v_\nu=\bv\cdot\bnu$ and $\bv_\tau=\bv-v_\nu\bnu$.  Moreover,
$\bvarepsilon(\bv)$ denotes the symmetric
part of the gradient of $\bv$, i.e.,
\[\bvarepsilon(\bv)=\frac{1}{2}\big(\nabla \bv+\nabla^T\bv\big).\]
In addition, for a regular tensor-valued field $\bsigma\colon \Omega\to\mathbb{S}^d$ we shall use $\sigma_\nu$ and $\bsigma_\tau$ for the normal and tangential components of the stress vector $\bsigma\bnu$ on $\Gamma$, i.e., $\sigma_\nu=\bsigma\bnu\cdot\bnu$ and $\bsigma_\tau=\bsigma\bnu-\sigma_\nu\bnu$.

Next, for the displacement field we need the space $V$ and for the stress and strain fields  we need the space $Q$, defined as follows:
\begin{eqnarray*}
&&\label{spV}V=\{\,\bv\in H^1(\Omega)^d:\  \bv =\bzero\ \ {\rm on\ \ }\Gamma_1 \,\},\\
&&\label{spQ}Q=\{\,\bsigma=(\sigma_{ij}):\ \sigma_{ij}=\sigma_{ji} \in L^{2}(\Omega)\ \ \ \forall\, i,\,j=1,\dots,d\,\}.
\end{eqnarray*}
The spaces $V$ and $Q$  are real Hilbert spaces endowed with the inner products
\begin{equation}\label{eeq}
(\bu,\bv)_V= \int_{\Omega}
\bvarepsilon(\bu)\cdot\bvarepsilon(\bv)\,dx,\qquad ( \bsigma,\btau )_Q =
\int_{\Omega}{\bsigma\cdot\btau\,dx}.
\end{equation}
The associated norms on these spaces will be denoted by $\|\cdot\|_{V}$ and $\|\cdot\|_{Q}$, respectively. Recall that the
completeness of the space $(V,\|\cdot\|_{V})$ follows from the
assumption ${meas}\,(\Gamma_1)>0$, which allows the use of Korn's
inequality.
Note also that, by the definition of the inner product in the spaces $V$ and $Q$, we have
\begin{equation}
\|\bv\|_V=\|\bvarepsilon(\bv)\|_Q\ \  \ \ \mbox{for all} \ \ \bv\in V
\label{684}
\end{equation}
and, using the Sobolev
theorem, we deduce that
\begin{equation}\label{trace}
\|\bv\|_{L^2(\Gamma)^d}\le c_{0}\,\|\bv\|_{V}\quad
{\rm for\ all}\  \ \bv \in V.
\end{equation}
Here, $c_{0}$ is a positive constant which depends on $\Omega$ and $\Gamma_1$.

We also use notation ${\bf Q_\infty}$  for the space of fourth order tensor fields defined by
\[{\bf Q_\infty}=\{\, {\cal C}=(c_{ijkl})\  :\
{c}_{ijkl}={c}_{jikl}={c}_{klij} \in L^\infty(\Omega) \ \ \ \forall\, i,\,j,\,k,\,l=1,\dots,d\,\},\] equipped with the norm
\begin{equation*}\label{**}
	\displaystyle \|{\cal{C}}\|_{\bf Q_{\infty}}=\max_{1\le i,j,k,l\le
		d}\|{c}_ {ijkl}\|_{L^{\infty}(\Omega)}. \end{equation*}

We end this section with the following result we shall use in  the rest of the paper.

\begin{lemma}\label{l1n} There exists a linear continuous operator
	$G\colon Q\to V$ such that for any\ $\bo\in Q$ and $\bu\in V$ the following implication hold:
	\begin{equation*}
	\label{50}
	\bo=\bvarepsilon(\bu)\quad\ \Longrightarrow\quad \bu=G\bo. \end{equation*}
\end{lemma}

{The proof of Lemma \ref{l1n}  is obtained by standard ortogonality arguments used in various books and surveys and, therefore, we skip it. Such arguments have been used  in \cite{Tem}, for instance, in the study of Navier--Stokes equations.}


\section{The viscoelastic contact model}\label{s3}
\setcounter{equation}0

We now describe the mathematical model of contact we consider in this paper. The physical setting is the following: a viscoelastic body occupies, in its reference configuration, a bounded domain $\Omega\subset \real^d$ ($d\in\{2,3\}$), with regular boundary $\partial\Omega=\Gamma$. We assume that
$\Gamma$ is
decomposed into three parts $\overline{\Gamma}_1$,
$\overline{\Gamma}_2$ and $\overline{\Gamma}_3$, with $\Gamma_1$,
$\Gamma_2$ and $\Gamma_3$ being relatively open and mutually
disjoint and, moreover,  the $d-1$ measure of $\Gamma_1$, denoted by ${meas}\,(\Gamma_1)$, is positive. The body is fixed on the  part $\Gamma_1$ of its boundary, is acted upon by body forces and surface tractions on $\Gamma_2$, and is in contact with an obstacle on $\Gamma_3$, the co-called foundation.
As a result, its mechanical state  evolves. To describe its evolution we denote by $[0,T]$ the time interval of interest, where  $T>0$.  Moreover, we use $\bx$ to denote a typical point in $\Omega\cup\Gamma$  and, for simplicity, we sometimes skip the dependence of various functions on the spatial variable $\bx$.
Then, the viscoelastic contact model we consider is as follows.

\medskip\noindent
{\bf Problem} ${\cal P}$. {\it Find a displacement field
	$\bu\colon\Omega\times[0,T] \to\mathbb{R}^d$,
	a stress field $\bsigma\colon\Omega\times [0,T]\to \mathbb{S}^d$  and a strain field $\bo\colon\Omega\times [0,T]\to \mathbb{S}^d$
	such that for any $t\in[0,T]$ the following hold: }
\begin{eqnarray}
	\label{1m}
	\bsigma(t)={\mathcal A}\bo(t)+\int_0^t {\cal B}(t-s)\bo(s)\,ds\quad
	\ &{\rm in}\ &\Omega,\\[1mm]
	\label{1n}
	\bo(t)=\bvarepsilon(\bu(t))
	\ &{\rm in}\ &\Omega,\\[2mm]
	\label{2m} {\rm Div}\,\bsigma(t)+\fb_{0}(t)=\bzero\quad&{\rm in}\ &\Omega,\\[2mm]
	\label{3m} \bu(t)=\bzero\quad  &{\rm on}\  &\Gamma_1,\\[2mm]
	\label{4m} \bsigma(t)\bnu=\fb_2(t)\quad&{\rm on}\ &\Gamma_2, \\ [3mm]
	\label{5m}  \left.\begin{array}{ll} \sigma_\nu(t) =0 \ \ \ \ \ \ \quad\quad&{\rm if}\ \ u_\nu(t) < 0\\[2mm]
	-F\le\sigma_\nu(t) \le 0 \ \ \ &{\rm if}\ \ \   u_\nu(t) = 0\\[2mm]
	\sigma_\nu(t) = -F \ \ \quad\quad &{\rm if}\ \ \ u_\nu(t)>0
\end{array}\right\}\quad
&{\rm on}&\ \Gamma_3, \\[3mm]
	\label{6m} \bsigma_\tau(t)=\bzero\quad&{\rm on}\ &\Gamma_3.
\end{eqnarray}

\medskip
A short description of the equations and boundary conditions in Problem $\cP$ is as follows.
First,  equality (\ref{1m})  is the viscoelastic constitutive law with long memory in which $\cA$ and $\cB$ are the elasticity and the relaxation tensors, respectively. {It was considered  in many books, including
\cite{Dz,DL,Pipkin}. In particular, existence and uniqueness results for displacement-tractions boundary value problems involving such a constitutive law have been considered in \cite{DL}.}
Equality (\ref{1n}) represents the definition of the strain tensor.
Next, equation (\ref{2m}) is the equilibrium equation in which $\fb_0$ denotes the time-dependent density of body forces. We use this equation here since we assume that the mechanical process is quasistatic and, therefore, we neglect
the inertial term in the equation of motion.
The boundary condition (\ref{3m}) is the displacement condition and models the setting when the body is held fixed on the part $\Gamma_1$ of its boundary.  Condition (\ref{4m})
is the traction boundary condition in which $\fb_2$ represents the density of surface tractions which act on $\Gamma_2$, assumed to be time-dependent. Condition (\ref{5m}) describes the contact with a~rigid-plastic foundation. It shows that when there is separation (i.e., when $u_\nu(t)<0$) then the reaction of the foundation vanishes (since $\sigma_\nu(t)=0$); moreover, it shows that penetration arise only if the normal stress reaches the value $F$, which is interpreted as the yield limit of the foundation. More details and mechanical interpretation on {this condition and similar interface laws could be found in \cite[p. 280]{S} and \cite{SM,SMBOOK}, for instance}.  Finally, condition (\ref{6m}) shows that the shear on the contact surface vanishes during the process. We use this condition here since we assume that the contact is frictionless. The case of a frictional contact problem can be considered and treated by using similar arguments, too. Nevertheless its analysis is more difficult since in the frictional case the function $j$ and the set $\Sigma$ we introduce below depend on the solution itself.

{
We end our comments on the model \eqref{1m}--\eqref{6m} with the remark that in the case when the memory term in \eqref{1m} vanishes (i.e., when $\cB\equiv \bzero$), then Problem $\cP$ reduces to a  time-dependent elastic contact problem. 
A comparison between the solution of this elastic problem and the original Problem $\cP$ has been made in  \cite[p. 301--302]{S}, under specific  assumptions.
For the example presented there, it was proved that the memory term does not affect the stress field but, in contrast,  it affects the strain and the displacement field. Moreover, the solution of the elastic contact problem can be obtained  from the solution of the viscoelastic contact problems, in the limit as the relaxation tensor converges to zero. We also mention that a comparison of the numerical solutions for the elastic and viscoelastic problems will be made in Section~6 (see Figure \ref{img::fp}).
}

In the study of Problem $\cP$ we  assume that the viscosity and the  elasticity operators satisfy the following conditions.

\begin{equation}
	\left\{\begin{array}{ll}
		{\rm (a)\ } {\cal A}=(a_{ijkl})\in
		{\bf Q_\infty}.\\ [2mm]
		{\rm (b)\ There\ exists}\ m_{\mathcal A}>0\ {\rm\ such\ that}\\
		{} \qquad  {\cal A}(\bx,\btau)\cdot\btau\ge m_{\mathcal A}
		\|\btau\|^2\ \,
		\forall\,\btau\in \mathbb{S}^d,\ {\rm a.e.}\ \bx \in\Omega.
	\end{array}\right.
	\label{Am}
\end{equation}

\begin{equation}\label{Bm}
	\cB\in C([0,T];{\bf Q_\infty}).
\end{equation}

\noindent
Moreover, the density of applied forces and the yield limit of the foundation have the regularity
\begin{eqnarray}
	&&\label{f0} \fb_0 \in C([0,T];L^2(\Omega)^d).\\[2mm]
	&&\label{f2}\fb_2\in C([0,T];L^2(\Gamma_2)^d). \\[2mm]
	&&\label{Fm} F\in L^2(\Gamma_3),\quad F(\bx)\ge 0\quad{\rm a.e. }\ \ \bx\in\Gamma_3.
\end{eqnarray}

We shall keep assumptions (\ref{Am})--(\ref{Fm}) in the next three sections, even if we do not mention it explicitly. Our main aim there is to provide the variational analysis of  Problem $\cP$, including  existence, uniqueness and convergence results.

\section{Variational formulations}\label{s4}
\setcounter{equation}0

In order to deduce the variational formulations for Problem $\cP$ we introduce  the operators $A\colon V\to V$, $\widetilde{A}\colon Q\to Q$, $\cS\colon C([0,T];V)\to C([0,T];V)$ and  $\widetilde{\cS}\colon C([0,T];Q)\to C([0,T];Q)$
defined by equalities

\begin{eqnarray}
&&\label{MA}(A\bu,\bv)_V =
	\int_{\Omega}{\cal A}\bvarepsilon(\bu)\cdot\bvarepsilon(\bv)\,dx
	\qquad\forall
	\,\bu,\,\bv\in V,\\[3mm]
&&\label{MAt}(\widetilde{A}\bo,\btau)_Q =
	\int_{\Omega}{\cal A}{\bo}\cdot\btau\,dx
	\qquad\forall
	\,\bo,\,\btau\in Q,\\ [3mm]
&&\label{MS}(\cS\bu(t),\bv)_V =
\int_{\Omega}\int_0^t{\cal B}(t-s)\bvarepsilon(\bu(s))\,ds\cdot\bvarepsilon(\bv)\,dx\\[1mm]
&&\qquad\qquad\forall
\,\bu\in C([0,T]; V),\,\bv\in V,\ t\in[0,T],\nonumber\\[3mm]
&&\label{MSt}(\widetilde{\cS}\bo(t),\btau)_Q =
\int_{\Omega}\int_0^t{\cal B}(t-s)\bo(s)\,ds\cdot\btau\,dx\\ [1mm]
&&\qquad\qquad\forall
\,\bo\in C([0,T]; Q), \,\btau\in Q,\ t\in[0,T].\nonumber
\end{eqnarray}

Using the assumptions on the elasticity tensor $\cA$, it is easy to see that $A\colon V\to V$ and $\widetilde{A}\colon Q\to Q$ are linear continuous symmetric and coercive operators.  Moreover,
it is easy to see that $\cS\colon C([0,T];V)\to C([0,T];V)$ and $\widetilde{\cS}\colon C([0,T];Q)\to C([0,T];Q)$ are history-dependent operators.
This allows us to use Theorem \ref{t0} on the space $X=Q$.
Below in this section we denote by $\widetilde{A}^{-1}\colon Q\to Q$ the inverse of the operator $\widetilde{A}$  and we use $\widetilde{\cR}\colon C([0,T];Q)\to C([0,T];Q)$ for the corresponding history-dependent operator.

Next, we consider the functions $j\colon V\to\real$, $\fb\colon [0,T]\to V$ and the set $\Sigma(t)$
defined by equalities
\begin{eqnarray}
	&&\hspace{-8mm}\label{Mj}j(\bv) =
	\int_{\Gamma_3} Fv_\nu^+\,da\qquad\forall
	\,\bv\in V,\\[2mm]
	&&\hspace{-8mm}\label{Mf}(\fb(t),\bv)_V = \int_{\Omega}\fb_0(t)\cdot\bv\,dx+
	\int_{\Gamma_2}\fb_2(t)\cdot\bv\,da\qquad\forall\,\bv\in V,\ t\in[0,T]\\[3mm]
	&&\hspace{-8mm}\label{Msm}\Sigma(t)=\{\,\btau\in Q\, :\, (\btau,\bvarepsilon(\bv))_Q
	+j(\bv)\ge (\fb(t),\bv)_V\ \ \forall\, \bv\in V \,\}\ \forall\, t\in[0,T].
\end{eqnarray}
Note that in (\ref{Mj}) and below, we use notation $r^+$ for the positive part of $r\in\mathbb{R}$, that is,  $r^+={\rm max} \,\{r,0\}$. Therefore, $j$ is a positively homogeneous function, i.e., $j(\lambda\bv)=\lambda j(\bv)$ for each $\lambda>0$ and $\bv\in V$.

Assume now that  $(\bu,\bsigma,\bo)$ are sufficiently regular functions which satisfy Problem~$\cP$.
We use (\ref{3m}), (\ref{1m}) and (\ref{1n}) to see that
\begin{equation}\label{MxA}
	\bu(t)\in V,\qquad\bsigma(t)\in Q,\qquad\bo(t)\in Q\qquad\forall\, t\in[0,T].
\end{equation}

Let $\bv\in V$ and $t\in[0,T]$. Then, using standard arguments based on integration by parts we deduce that
\begin{eqnarray*}
	&&\label{Mx1}\hspace{-6mm}
	\int_\Omega\,\bsigma(t)\cdot(\bvarepsilon(\bv)-\bvarepsilon(\bu(t)))\,dx+
	\int_{\Gamma_3}Fv_\nu^+\,da-\int_{\Gamma_3}Fu_\nu^+(t)\,da\\
	&&\qquad\qquad\ge\int_{\Omega}\fb_0(t)\cdot(\bv-\bu(t))\,dx+\int_{\Gamma_2}\fb_2(t)\cdot(\bv-\bu(t))\,da.\nonumber
\end{eqnarray*}

\noindent
Next, we use notation
(\ref{Mj}) and (\ref{Mf}) to deduce that
\begin{equation}\label{Mx2}
	(\bsigma(t),\bvarepsilon(\bv)-\bvarepsilon
	(\bu(t)))_Q+j(\bv)-j(\bu(t))\ge (\fb(t),\bv-\bu(t))_V.
\end{equation}
We now use the constitutive law (\ref{1m}) and notation  (\ref{MA}), (\ref{MS}) to see that
\begin{equation}\label{Mx2n}
(\bsigma(t),\bvarepsilon(\bv)-\bvarepsilon(\bu(t)))_Q=(A\bu(t),\bv-\bu(t))_V+(\cS\bu(t),\bv-\bu(t))_V.
\end{equation}
Therefore, substituting (\ref{Mx2n}) in (\ref{Mx2}) and using (\ref{MxA}) we deduce the following variational formulation of the contact Problem $\cP$ in terms of displacement.

\medskip\noindent
{\bf Problem} ${\cal P}^V_1$. {\it Find a displacement field $\bu\in C([0,T];V)$
such that for all $t\in[0,T]$ the following inequality holds:}
\begin{eqnarray}\label{Mx8}
&&(A\bu(t),\bv-\bu(t))_V+(\cS\bu(t),\bv-\bu(t))_V+j(\bv)-j(\bu(t))\\ [2mm]
&&\qquad\qquad\ge (\fb(t),\bv-\bu(t))_V\qquad\forall\, \bv\in V.\nonumber
\end{eqnarray}

We now consider the following two variational formulations of Problem
$\cP$, in terms of the stress and strain field, respectively.

\medskip\noindent
{\bf Problem} ${\cal P}^V_2$. {\it Find a stress field $\bsigma\in C([0,T];Q)$
	such that for all $t\in[0,T]$ the following inequality holds:}
\begin{equation}\label{Mx9}
	\bsigma(t)\in \Sigma(t),\quad (\widetilde{A}^{-1}\bsigma(t),\btau-\bsigma(t))_Q+(\widetilde{\cR}\bsigma(t),\btau-\bsigma(t))_Q\ge 0\quad\ \forall\, \btau\in \Sigma(t).
\end{equation}

\medskip\noindent
{\bf Problem} ${\cal P}^V_3$. {\it Find a strain field $\bo\in C([0,T];Q)$
	such that  for all $t\in[0,T]$ the following inclusion holds:}
\begin{equation}\label{Mx10}
	-\bo(t)\in N_{\Sigma(t)}(\widetilde{A}\bo(t)+\widetilde{\cS}\bo(t)).
\end{equation}

Note that inequality (\ref{Mx9}) and inclusion
(\ref{Mx10}) can be derived directly from the statement of the contact Problem $\cP$. Nevertheless, to avoid repetitions we do not provide this derivation, and we restrict ourselves to mention that Problems ${\cal P}^V_2$ and ${\cal P}^V_3$ are fully justified by the following results.

\begin{proposition}\label{p1} Let $\bu$ be a solution of Problem ${\cal P}^V_1$ and let $\bsigma\colon [0,T]\to Q$ be the function defined by equality
\begin{equation}\label{v1}
\bsigma(t)=\widetilde{A}\bvarepsilon(\bu(t))+\widetilde{\cS}\bvarepsilon(\bu(t))\quad\forall\, t\in[0,T].
\end{equation}
Then $\bsigma$ is a solution of Problem ${\cal P}^V_2$.
\end{proposition}

\begin{proof} The regularity $\bsigma\in C([0,T];Q)$ is obvious. Moreover, using (\ref{v1})
and Theorem \ref{t0} we deduce that
\begin{equation}\label{v1n}
\bvarepsilon(\bu(t))=	\widetilde{A}^{-1}\bsigma(t)+\widetilde{\cR}\bsigma(t)\quad\forall\, t\in[0,T].
\end{equation}

Let $\bv\in V$ and $t\in[0,T]$.
We use definitions (\ref{MA})--(\ref{MSt}) and (\ref{v1}) to see that
\[(A\bu(t),\bv-\bu(t))_V+(\cS\bu(t),\bv-\bu(t))_V=(\bsigma(t),\bvarepsilon(\bv)-\bvarepsilon(\bu(t)))_Q\]
and, therefore, (\ref{Mx8}) implies that
\begin{equation}\label{v2}
(\bsigma(t),\bvarepsilon(\bv)-\bvarepsilon(\bu(t)))_Q+j(\bv)-j(\bu(t))\ge (\fb(t),\bv-\bu(t))_V\qquad\forall\, \bv\in V.
\end{equation}
We now test in  (\ref{v2}) with  $\bv=2\bu(t)$ and $\bv=\bzero_{V}$ to see that
	\begin{equation}\label{v3}
		(\bsigma(t),\bvarepsilon(\bu(t)))_Q+j(\bu(t))=(\fb(t),\bu(t))_V.
	\end{equation}
	Therefore, using (\ref{v2}) and (\ref{v3}) we find that
	\[
	(\bsigma(t),\bvarepsilon(\bv))_Q+j(\bv)\ge (\fb(t),\bv)_V.\]
	This inequality combined with definition (\ref{Msm}) implies
	that
	\begin{equation*}\label{v4}
		\bsigma\in \Sigma(t).
	\end{equation*}
	To proceed, we use (\ref{Msm}), (\ref{MxA}) and (\ref{v3}) to see that
	\begin{eqnarray*}
		(\btau-\bsigma(t),\bvarepsilon(\bu(t)))_Q\ge 0 \qquad\forall\, \btau\in \Sigma(t)
	\end{eqnarray*}
	and, using (\ref{v1n})  we find that
	\begin{equation*}\label{Mx5}
		(\btau-\bsigma(t),\widetilde{A}^{-1}\bsigma(t)+\widetilde{\cR}\bsigma(t))_Q\ge 0 \qquad\forall\, \btau\in \Sigma(t).
	\end{equation*}
This shows that $\bsigma$ is a solution to Problem $\cP_2^V$, which concludes the proof.
\end{proof}

\begin{proposition}\label{p2} Let $\bsigma$ be a solution of Problem ${\cal P}^V_2$ and let $\bo\colon [0,T]\to Q$ be the function defined by equality
\begin{equation}\label{v1m}
	\bo(t)=\widetilde{A}^{-1}\bsigma(t)+\widetilde{\cR}\bsigma(t)\quad\forall\, t\in[0,T].
\end{equation}
Then $\bo$ is a solution of Problem ${\cal P}^V_3$.
\end{proposition}

\begin{proof}
The regularity $\bo\in C([0,T];Q)$ is obvious. Moreover, using (\ref{v1m}) and  Theorem \ref{t0}
we find that
\begin{equation}\label{v1q}
	\bsigma(t)=\widetilde{A}\bo(t)+\widetilde{\cS}
	\bo(t)\quad\forall\, t\in[0,T].
\end{equation}

Let $t\in[0,T]$. Then, using (\ref{Mx9}), (\ref{v1m}) and (\ref{v1q}) we obtain that
\begin{equation}\label{Mx9n}
	\widetilde{A}\bo(t)+\widetilde{\cS}
	\bo(t)\in \Sigma(t),\quad (\widetilde{A}\bo(t)+\widetilde{\cS}
	\bo(t)-\btau,\bo(t))_Q\le 0\quad\ \forall\, \btau\in \Sigma(t).
\end{equation}
Then, with  notation $N_{\Sigma(t))}$ for the outward normal cone of  the set $\Sigma(t)\subset Q$, equivalence (\ref{5}) and inequality (\ref{Mx9n}) imply that
\begin{equation*}
	-\bo(t)\in N_{\Sigma(t)}(\widetilde{A}\bo(t)+\widetilde{\cS}\bo(t)).
\end{equation*}
We conclude from here that $\bo$ is a solution to Problem $\cP_3^V$, which ends the proof.
\end{proof}

\begin{proposition}\label{p3} Let $\bo$ be a solution of Problem ${\cal P}^V_3$. Then there exists a function  $\bu\colon [0,T]\to V$ such that
\begin{equation}\label{v1p}
		\bo(t)=\bvarepsilon(\bu(t)) \quad\forall\, t\in[0,T].
	\end{equation}
Moreover, $\bu$ is a solution of Problem ${\cal P}^V_1$.

\end{proposition}

\begin{proof}  Let $\bsigma\colon [0,T]\to Q$ be the function defined by (\ref{v1q}) and let $t\in[0,T]$ be fixed. Then, using (\ref{Mx10}) and equivalence (\ref{5}) we deduce that
\begin{equation}\label{v12}
\bsigma(t)\in\Sigma(t),\qquad(\btau-\bsigma(t),\bo(t))_Q\ge 0 \qquad\forall\, \btau\in \Sigma(t).
\end{equation}
Next, consider an element $\bz\in Q$ such that
\begin{equation}\label{v13}
(\bz,\bvarepsilon(\bv))_Q=0\qquad\forall\bv\in V.
\end{equation}
Then, the definition (\ref{Msm}) implies that	$\bsigma(t)\pm \bz\in\Sigma(t)$ and, testing with $\btau=\bsigma(t)\pm \bz$ in (\ref{v12})
we find that
\begin{equation}\label{v14}	(\bo(t),\bz)_Q=0.
\end{equation}
Equalities (\ref{v13}) and (\ref{v14}) show that $\bo(t)\in (\bvarepsilon(V)^\perp)^\perp$	 where $M^\perp$ denotes the orthogonal of the set $M$ in the Hilbertian  structure of the space $Q$. Now, since $\bvarepsilon(V)$ is a closed subspace of $Q$ we have $(\bvarepsilon(V)^\perp)^\perp=\bvarepsilon(V)$.
We conclude from here that $\bo(t)\in \bvarepsilon(V)$ which shows that there exists an element $\bu(t)\in V$ such that (\ref{v1p}) holds.

It is easy to see that the function $t\mapsto \bu(t):[0,T]\to V$ defined above is continuous, i.e., $\bu\in C([0,T];V)$. We now prove that $\bu$ satisfies inequality (\ref{Mx8}).
To this end, let $t\in[0;T]$. We combine (\ref{v1p}) and (\ref{v12})
to see that
\begin{equation}\label{M22n}
	(\btau-\bsigma(t),\bvarepsilon(\bu(t)))_Q\ge 0\qquad\forall\,\btau\in\Sigma(t).
\end{equation}
Recall now that the function $j\colon V\to\mathbb{R}$ is subdifferentible on $V$, which
allows us to consider an element $\bxi(t)\in V$ such that
\begin{equation}\label{v13n}
j(\bv)-j(\bu(t))\ge (\bxi(t),\bv-\bu(t))_V\qquad\forall\,\bv\in V.
\end{equation}
Let $\btau_0(t)=\bvarepsilon(\fb(t)-\bxi(t))\in Q$.
Then using (\ref{eeq}) and (\ref{v13n})
it is easy to see that
\begin{equation}\label{v14n}
	(\btau_0(t),\bvarepsilon(\bv)-\bvarepsilon(\bu(t)))_Q+j(\bv)-j(\bu(t))\ge (\fb(t),\bv-\bu(t))_V\quad\forall\,\bv \in V.
\end{equation}
We now test in  (\ref{v14n}) with  $\bv=2\bu(t)$ and $\bv=\bzero_{V}$ to see that
\begin{equation}\label{v15}
	(\btau_0(t),\bvarepsilon(\bu(t)))_Q+j(\bu(t))=
	(\fb(t),\bu(t))_V.
\end{equation}
Therefore, using (\ref{v14n}) and (\ref{v15}) we find that
\[
(\btau_0(t),\bvarepsilon(\bv))_Q+j(\bv)\ge (\fb(t),\bv)_V\qquad\forall\,\bv\in V\]
which implies that $\btau_0(t)\in \Sigma(t)$. This regularity allows us to test with $\btau=\btau_0(t)$ in (\ref{M22n})
in order to see  that
\begin{equation*}
	(\btau_0(t),\bvarepsilon(\bu(t)))_Q+j(\bu(t))\ge (\bsigma(t),\bvarepsilon(\bu(t)))_Q+j(\bu(t))
\end{equation*}
and, using (\ref{v15}), we deduce that
\begin{equation}\label{Mx4n}
	(\fb(t),\bu(t))_V\ge (\bsigma(t),\bvarepsilon(\bu(t)))_Q+j(\bu(t)).
\end{equation}
On the other hand, since $\bsigma(t)\in\Sigma(t)$ we find that
\begin{equation}\label{Mx4m}
	(\bsigma(t),\bvarepsilon(\bv))_Q+j(\bv)\ge (\fb(t),\bv)_V\quad\forall\,\bv\in V
\end{equation}
which, in particular, implies that
\begin{equation}\label{Mx4q}
	(\bsigma(t),\bvarepsilon(\bu(t)))_Q+
	j(\bu(t))\ge (\fb(t),\bu(t))_V\quad\forall\,\bv\in V.
\end{equation}
We now combine inequalities  (\ref{Mx4n}) and (\ref{Mx4q}) to obtain that
\begin{equation}\label{Mx4p}
	(\bsigma(t),\bvarepsilon(\bu(t)))_Q+j(\bu(t))=
	(\fb(t),\bu(t))_V\quad\forall\,\bv\in V,
\end{equation}
then we use  (\ref{Mx4m}) and  (\ref{Mx4p}) to find that (\ref{Mx2}) holds, for each $\bv\in V$. 
Moreover, we observe that (\ref{v1q}), (\ref{v1p}) and definitions (\ref{MA}) -- (\ref{MSt}) of the operators $A$, $\widetilde{A}$, $\cS$ and $\widetilde{\cS}$, respectively, imply that
\begin{equation*}
(\bsigma(t),\bvarepsilon(\bv)-\bvarepsilon(\bu(t)))_Q=(A\bu(t)+\cS\bu(t),\bv-\bu(t))_V.
\end{equation*}
We substitute this equality in (\ref{Mx2}) and deduce that (\ref{Mx8}) holds.
This implies that $\bu$ is a solution of Problem $\cP_1^V$ and concludes the proof.
\end{proof}

{We now end this section with two remarks concerning the variational problems $\cP_1^V$, $\cP_2^V$ and $\cP_3^V$. The first one (Remark \ref{r3}
below) is mathematical in nature; the second one (Remark \ref{r4}) is mechanical in nature.
}

\begin{remark}\label{r3}
	We now follow \cite{RS1,S} to recall the following definition: two abstract Problems ${\cal P}$ and ${\cal Q}$ defined on the normed spaces $X$ and $Y$, respectively,  are said to be {\rm dual of each other} if there exists an operator $D\colon X\to Y$ such that:

\begin{enumerate}
  \item[\rm (a)]$D$ is bijective;
  
  \item[\rm (b)] Both $D\colon X\to\ Y$ and its inverse $D^{-1}\colon Y\to X$  are continuous;
  
  \item[\rm (c)] $u\in X$ is a solution of Problem ${\cal P}$ if and only if $\sigma:=Du\in Y$ is a solution of Problem ${\cal Q}$.
\end{enumerate}

\noindent
Then, it follows from Propositions $\ref{p1}$--$\ref{p3}$ that Problems $\cP_1^V$, $\cP_2^V$ and $\cP_3^V$ are pairwise dual of each other.
\end{remark}

{\begin{remark}\label{r4}
The arguments presented at the beginning of this section show that if the triple ($\bu,\bsigma, \bo$) represents a solution to Problem~$\cP$, then $\bu$ is a solution of Problem~$\cP_1^V$. This allows us to consider Problem $\cP^V_1$ as a (first) variational formulation of the contact Problem $\cP$.
Moreover,  notations \eqref{MAt} and  \eqref{MSt}  show that equality \eqref{v1} is equivalent with the constitutive law
\begin{equation}\label{v1qz}
	\bsigma(t)={\cA}\bvarepsilon(\bu(t))+\int_0^t{\cB}(t-s)
	\bvarepsilon(\bu(s))\,ds\quad\forall\, t\in[0,T]
\end{equation}
which, obviously, represents a consequence of equalities \eqref{1m} and \eqref{1n}.
Therefore, Proposition $\ref{p1}$ shows that if $\bu$ is a solution of Problem $\cP_1^V$ and $\bsigma$ is defined by the constitutive law $\eqref{v1qz}$, then $\bsigma$ is a solution of Problem $\cP_2^V$. This allows us to consider Problem $\cP^V_2$ as a (second) variational formulation of the contact Problem~$\cP$.
Next,  Theorem  $\ref{t0}$ and equality \eqref{v1p} imply that notation \eqref{v1m} is equivalent with the constitutive law \eqref{v1} which, in turn, is equivalent
with the constitutive law \eqref{v1qz}, as proved above.
Therefore, Proposition $\ref{p3}$ shows that if $\bsigma$ is a solution of Problem~$\cP_2$ and $\bo$ is defined by $\eqref{v1m}$, then $\bo$ is a solution of Problem $\cP_3^V$. This allows us to consider Problem $\cP^V_3$ as a (third) variational formulation of the contact Problem~$\cP$.
The arguments above provide the legitimacy of weak formulations  $\cP^V_1$, $\cP_V^2$ and $\cP^V_3$. These formulations are expressed in terms of different unknowns and have a different structure. Nevertheless, each one can be considered as a variational formulation of the original contact problem $\cP$. 
\end{remark} 
}

\section{Weak solvability}\label{s5}
\setcounter{equation}0

In this section we turn to the solvability of the variational Problems $\cP_1^V$, $\cP_2^V$ and $\cP_3^V$. Our main result on this matter is the following.

\begin{theorem}\label{Mt2}
	Assume  $(\ref{Am})$--$(\ref{Fm})$. Then, Problems $\cP_1^V$, $\cP_2^V$ and $\cP_3^V$ have a unique solution. Moreover, the solution depends  Lipschitz continuously on the data $(F,\fb)\in L^2(\Gamma_3)\times C([0,T];V)$.
\end{theorem}

\begin{proof}  Using assumption (\ref{Am}) it is easy to see that the operator $A\colon V\to V$ defined by (\ref{MA}) is a strongly monotone Lipschitz continuous operator. Moreover, recall that the operator $\cS\colon C([0,T];V)\to C([0,T];V)$ given by (\ref{MS}) is a history-dependent operator. In addition, assumption (\ref{Fm}) guarantees that the function $j\colon V\to\real$ defined by (\ref{Mj}) is a continuous seminorm and, therefore, it is convex and lower semicontinuous.  Finally, the regularities (\ref{f0}), (\ref{f2}) imply that $\fb\in C([0,T];V)$.  Therefore, we are in a~position
to apply Theorem \ref{t1} with $K=X=V$. In this way we prove the existence of a unique solution to Problem $\cP_1^V$.

Next, using Propositions \ref{p1} and \ref{p2} we deduce the solvability of Problems $\cP_2^V$ and $\cP_3^V$, respectively. To prove their uniqueness, we proceed as follows. Assume that $\bo$ and $\widetilde{\bo}$ represent two solutions of Problem $\cP_3^V$. Then Proposition
\ref{p3} shows that there exist two functions  $\bu,\,\widetilde{\bu}\colon [0,T]\to V$ such that
\begin{equation}\label{v1pp}
	\bo(t)=\bvarepsilon(\bu(t)),\quad \widetilde{\bo}(t)=\bvarepsilon(\widetilde{\bu}(t)) \quad\forall\, t\in[0,T].
\end{equation}
Moreover, $\bu$ and $\widetilde{\bu}$ are solutions of Problem ${\cal P}^V_1$. Now, using the uniqueness of the solution of  Problem ${\cal P}^V_1$ we deduce that $\bu=\widetilde{\bu}$ and, therefore (\ref{v1pp}) implies that $\bo=\widetilde{\bo}$. This proves the uniqueness of  the solution of Problem $\cP_3^V$.

Similarly, assume that $\bsigma$ and $\widetilde{\bsigma}$ represent two solutions of Problem $\cP_2^V$. Then Proposition
\ref{p2} combined with the uniqueness of the solution of Problem $\cP_3^V$ show that
\begin{equation*}
\widetilde{A}^{-1}\bsigma(t)+\widetilde{\cR}\bsigma(t)=\widetilde{A}^{-1}\widetilde{\bsigma}(t)+\widetilde{\cR}\widetilde{\bsigma}(t)\qquad\forall\, t\in[0,T].
\end{equation*}
Then the inversibility of the operator $\widetilde{A}^{-1}+\widetilde{\cR}\colon C([0,T];Q)\to
C([0,T];Q)$, guaranteed by Theorem \ref{t0}, implies that $\bsigma=\widetilde{\bsigma}$.
This proves the uniqueness of the solution of Problem $\cP_2^V$.

Assume now that  $(F_1,\fb^1),\,(F_2,\fb^2)\in L^2(\Gamma_3)\times C([0,T];V)$ and denote by
$\bu_i\in C([0,T];V)$ the solution of inequality (\ref{Mx8}) for $F=F_i$ and $\fb=\fb^i$, $i=1,2$. Then, for any $t\in[0,T]$ and $\bv\in V$ we have
	\begin{eqnarray}
		&&\label{Mx11}(A\bu_1(t),\bv-\bu_1(t))_V+
		(\cS\bu_1(t),\bv-\bu_1(t))_V\\ [2mm]
		&&\qquad\qquad+\int_{\Gamma_3}F_1\,v_\nu^+\,da-\int_{\Gamma_3}F_1\,u_{1\nu}^+(t)\,da\ge (\fb^1(t),\bv-\bu_1(t))_V,\nonumber\\ [5mm]
		&&\label{Mx12}(A\bu_2(t),\bv-\bu_2(t))_V+
		(\cS\bu_2(t),\bv-\bu_2(t))_V\\ [2mm]
		&&\qquad\qquad+\int_{\Gamma_3}F_2\,v_\nu^+\,da-\int_{\Gamma_3}F_2\,u_{2\nu}^+(t)\,da\ge (\fb^2(t),\bv-\bu_2(t))_V. \nonumber
	\end{eqnarray}
	We take $\bv=\bu_2(t)$ in (\ref{Mx11}),   $\bv=\bu_1(t)$ in (\ref{Mx12}), then we add the resulting inequalities to find that
	\begin{eqnarray*}
		&&(A\bu_1(t)-A\bu_2(t),\bu_1(t)-\bu_2(t))_V\le (\cS\bu_1(t)-\cS\bu_2(t),\bu_2(t)-\bu_1(t))_V\\ [2mm]
		&&\quad+\int_{\Gamma_3}(F_1-F_2)(u_{2\nu}^+(t)-u_{1\nu}^+(t))\,da+
		(\fb^1(t)-\fb^2(t),\bu_1(t)-\bu_2(t))_V.
	\end{eqnarray*}
	Next, the strong monotonicity of $A$ and the trace inequality (\ref{trace})	yield
	\begin{eqnarray*}
		&&m_{\cal A}\|\bu_1(t)-\bu_2(t)\|^2_V\le
		\|\cS\bu_1(t)-\cS\bu_2(t)\|_V\|\bu_1(t)-\bu_2(t)\|_V
		\\ [2mm]
		&&\quad+ c_0\|F_1-F_2\|_{L^2(\Gamma_3)}\|\bu_1(t)-\bu_2(t)\|_V+
		\|\fb^1(t)-\fb^2(t)\|_V\|\bu_1(t)-\bu_2(t)\|_V.
	\end{eqnarray*}
The previous inequality implies that
\begin{eqnarray*}
	&&m_{\cal A}\|\bu_1(t)-\bu_2(t)\|_V\le\|\cS\bu_1(t)-\cS\bu_2(t)\|_V
	\\ [2mm]
	&&\qquad\qquad+c_0\|F_1-F_2\|_{L^2(\Gamma_3)}+
	\|\fb^1(t)-\fb^2(t)\|_V.
\end{eqnarray*}

We now use definition (\ref{MS}) and assumption (\ref{Bm}) to see that there exists $c_1>0$ such that\begin{eqnarray*}
	&&m_{\cal A}\|\bu_1(t)-\bu_2(t)\|_V\le c_1\int_0^t\|\bu_1(s)-\bu_2(s)\|_V\,ds
	\\ [2mm]
	&&\qquad\qquad+c_0\|F_1-F_2\|_{L^2(\Gamma_3)}+
	\|\fb^1(t)-\fb^2(t)\|_V.
\end{eqnarray*}

	Next, by using a Gronwall argument and definition (\ref{cann})
	we see that there exists $C>0$ which does not depend on $F_i$ and $\fb^i$, $i=1,2$, such that
	\[\|\bu_1-\bu_2\|_{C([0,T];V)}\le C\big(\|F_1-F_2\|_{L^2(\Gamma_3)}+\|\fb^1-\fb^2\|_{C([0,T];V)}\big)\]
	which shows that the solution $\bu\in C([0,T];V)$ depends  Lipschitz continuously on the data $(F,\fb)\in L^2(\Gamma_3)\times C([0,T];V)$.
	The Lipschitz continuity of the solutions $\bsigma$ and $\bo$ follows now from equalities (\ref{v1}) and (\ref{v1m})
	and the properties of the operators $A$, $\cS$,
	$\widetilde{A}^{-1}$ and $\widetilde{\cR}$.
\end{proof}

\begin{remark}\label{r1}
	Note that the unique solvability of Problem $\cP_2^V$ can be obtained directly. A sketch of the proof is as follows. First, note that
	\begin{equation*}
	\Sigma(t)= \Sigma_0+\bvarepsilon({\fb}(t))\quad\forall\, t\in[0,T],
\end{equation*}
where $\Sigma_0$ is the time-independent nonempty closed convex subset of $Q$ defined by
\[\Sigma_0=\{\,\btau\in Q\, :\, (\btau,\bvarepsilon(\bv))_Q
+j(\bv)\ge 0\ \ \forall\, \bv\in V \,\}.
\]
Then, using the change of unknown given by $\bsigma=\overline{\bsigma}+\bvarepsilon({\fb})$,
we find that Problem $\cP_2^V$ is equivalent with a history-dependent variational inequality of the form $(\ref{zz})$ on the space $X=Q$,  associated to the convex $K=\Sigma_0$, in which the unknown is the auxiliary stress field $\overline{\bsigma}$.
Theorem $\ref{t1}$ guarantees the unique solvability of this inequality which, in turn, provides the unique solvability of Problem $\cP_2^V$.
\end{remark}

\begin{remark}\label{r2}
The unique solvability of Problem $\cP_3^V$ can be obtained directly, by using Theorem $\ref{t2}$.
Indeed, it is easy to check that assumptions
$(\mathcal{A})$, $(\mathcal{S})$ and $(\Sigma)$  are satisfied for the inclusion $(\ref{Mx10})$
with $X=Q$, operators $\widetilde{A}^{-1}$, $\widetilde{R}$ and the function $g=\bvarepsilon(\fb)$.
\end{remark}

We end this section with the remark  that (\ref{Mx8}), (\ref{Mx9}) represent history-de\-pen\-dent inequalities and
(\ref{Mx10}) is a history-dependent inclusion.
Despite the fact that these problems have a different structure, each of them can be interpreted as a variational formulation of the contact Problem $\cP$. We conclude from here that the variational formulation of contact models is not unique and could lead to different mathematical problems which, in fact, are dual of each other. Moreover, anyone among the displacement, the stress or the strain field can be considered as main unknown of the corresponding contact model, provided that an appropriate variational formulation is used.

We refer to a  triple $(\bu,\bsigma,\bo)$ such that $\bu$ is a solution of Problem $\cP_1^V$, $\bsigma$ is a~solution of Problem $\cP_2^V$ and $\bo$ is a solution of Problem $\cP_3^V$, as a weak solution to the contact Problem $\cP$.
We note that  Theorem \ref{Mt2} provides the unique weak solvability of Problem $\cP$ as well as the Lipschitz continuous dependence of the weak solution
with respect to the data $\fb$ and  $F$.

{Moreover, using standard arguments and Remark \ref{r4} it can be proved that if any of the solution to Problems $\cP^V_1$, $\cP^V_2$ or $\cP^V_3$ is smooth enough, then  the weak solution satisfies the equations and boundary conditions \eqref{1m}--\eqref{6m} in the strong sense, i.e. at each point $\bx\in\Omega$ and at any time moment $t\in[0,T]$.}

\section{Numerical approximation}\label{s7}
\setcounter{equation}0

In this section, we present numerical simulations for the contact Problem~$\cP$ by using its variational formulation given in Problem~$\cP_1^V$.
Throughout the rest of this paper, we assume that (\ref{Am})--(\ref{Fm}) hold, even if we do not mention it explicitly. By virtue of Theorem \ref{Mt2}, we conclude that Problem $\cP_1^V$ has a unique solution $\bu\in C([0,T];V)$.
We start by introducing a fully-discrete scheme to approximate the solution of Problem $\cP_1^V$.

Let $V^h$ be a~finite-dimensional subspace of $V$, where $h$ is positive real number which denotes the spatial discretization step.
Throughout this section, we assume that $V^h$ is the space of piecewise affine continuous functions, given by
\begin{equation*}
V^h=\{w^h\in C(\bar{\Omega};\real^d)\,\,\,|\,\, \,w^h|_{\widetilde{T}}\in [\mathbb{P}_1(\widetilde{T})]^d
\,\,\,\,\,{\rm for\,\, all}\,\,\widetilde{T}\in {\mathcal{T}}^h\}\subset V.
\label{Vh}
\end{equation*}

\noindent
Here, $\mathcal{T}^h$ is a family of finite element partitions of $\Omega$ and $\mathbb{P}_1(\widetilde{T})$ denotes the space of affine functions on $\widetilde{T}$.
We divide the time interval $[0, T]$ into $N$ equal pieces of length $k=\frac{T}{N}$ and  we denote $t_n = k\, n$, for all $n=1, 2, \dots, N$. Moreover, for a continuous function $g=g(t)$ we use the short-hand notation $g_n = g(t_n)$, for all $n=1, 2, \dots, N$.

We now can introduce the  following discrete version of Problem $\cP_1^V$.

\noindent
\textbf{Problem} $\cP_1^{Vh}$. {\it Find 
a displacement $\bu^{hk} = \{\bu^{hk}_i\}_{i=1}^N \subset V^h$ such that
the inequality below holds:}
\begin{eqnarray}\label{P1_disc}
&&(A\bu^{hk}_i,\bv^h-\bu^{hk}_i)_V+((\cS\bu^{hk})_i\,,\bv^h-\bu^{hk}_i)_V+j(\bv^h)-j(\bu^{hk}_i)\\ [2mm]
&&\qquad\qquad\ge (\fb_i , \bv^h-\bu^{hk}_i)_V\qquad\forall\, \bv^h\in V^h, \ i= 1,\dots,N.\nonumber
\end{eqnarray}

\medskip
\noindent
The unique solvability of the discrete Problem $\cP_1^{Vh}$ can be easily proved by using arguments similar to those used in the proof of {Theorem \ref{Mt2}}.

To present the numerical solution of problem (\ref{P1_disc}), we utilize the mechanical example based on the two-dimensional physical setting shown in Figure~1, which represents the cross-section of a three dimensional viscoelastic body. We note this section by $\Omega\subset\mathbb{R}^2$.
The body
is clamped on the part $\Gamma_1 = [0\,\unit{m}, 1\,\unit{m}]  \times \{ 0\,\unit{m} \}$ and, therefore, the displacement field vanishes there. On the part $\Gamma_3 = [4\,\unit{m}, 5\,\unit{m}] \times \{ 0\,\unit{m} \}$ the body is in potential frictionless contact with a~rigid-plastic penetrable foundation with the yield limit $F$.
Moreover, it is acted from the top by a vertical force of density $\fb_2$. Therefore, denoting by $\Gamma_2$ the remaining part of the boundary of $\Omega$  and using notation $\mathcal{C}^u((x_0, y_0), \, r)$\ for the upper semicircle of radius $r$ ($y \ge y_0$)  centred at the point $(x_0, y_0)$,
for any $\bx=(x,y)\in\Gamma_2$ and $t\in[0,T]$ we have
\begin{align*}
	&\fb_2((x, y), t)= \left \{ \begin{array}{ll}
		(0, f_{2 \, y}(t))\, \unit{N \, m^{-2}}, &  \text{if } (x, y) \in \mathcal{C}^u((2.5\,\unit{m}, 2.0\,\unit{m}), \, 2.5\,\unit{m}), \\
		(0, 0)\, \unit{N \, m^{-2}}, & \text{otherwise}.
	\end{array} \right.
\end{align*}

\begin{figure}[t]
\centering
    \hspace*{-1cm}
    \includegraphics[width=0.5\linewidth]{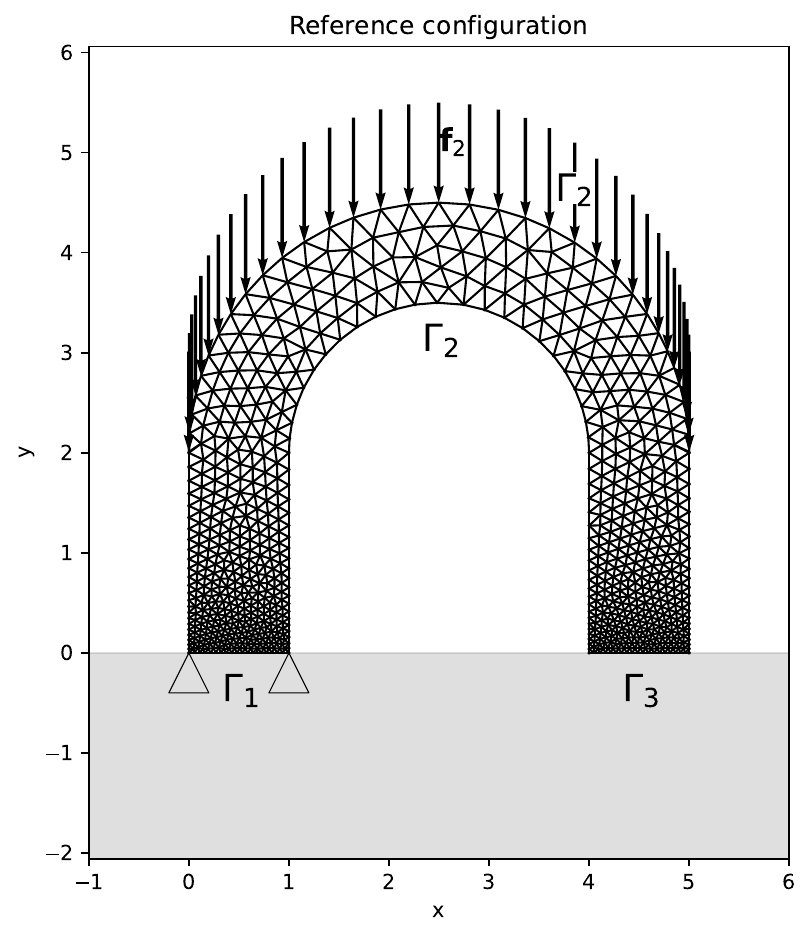}
 \caption{Reference configuration of the body.} \label{img::ref}
\end{figure}

For simplicity of the analysis, we neglect the body forces and, therefore, we assume that $\fb_0=\bzero$.
Moreover, we assume that the body behaves linearly and the components of elasticity and relaxation tensors are given by
\begin{eqnarray}
({\mathcal A}\bo)_{ij}& = &\frac{E \kappa}{(1 + \kappa)(1 - 2\kappa)}(\omega_{11} + \omega_{22})\delta_{ij} \, + \, \frac{E}{1 + \kappa}\omega_{ij} \label{elastic_tensor}\\
&&\qquad\qquad\qquad\qquad\forall
\, \bo=(\omega_{ij}) \in \mathbb{S}^2, \ i,j = 1, 2,  \nonumber\\[2mm]
({\mathcal B}(t) \, \bo)_{ij}& =& b \, \omega_{ij} \qquad  \forall \, \bo=(\omega_{ij}) \in \mathbb{S}^2, \ i,j = 1, 2, \quad t \in [0, T]. \label{relaxation_tensor}
\end{eqnarray}
\noindent
Here and below $\delta_{ij}$ is the Kronecker delta, $b$ is a relaxation parameter, and $E$ and $\kappa$ are Young's modulus and Poisson's ratio of the body material, respectively. For the simulations we present below, we use
the following input parameters:
\begin{align*}
 E &= 10^4 \,\unit{N \, m^{-2}}, \quad \kappa = 0.4, \\
F &= 10, \quad b = 10^4 \, \unit{N\, m^{-2} s^{-1}},\\
f_{2 \, y}(t) &= 10 \sin t,  \quad t \in [0,T].\\
\end{align*}

\vspace{-10mm}
In order to obtain a numerical solution, a spatial discretization with variable mesh size is used, with a maximum size of 0.275\,m inside the domain $\Omega$ and not exceeding 0.06\,m for elements lying directly on $\Gamma_1$ and $\Gamma_3$. This gives rise to a spatial domain discretized into 822 elements, including 20 contact elements, the total number of degrees of freedom being 1644.

To find a solution of the discrete variational inequality~(\ref{P1_disc}) with linear elasticity and relaxation tensors defined as in (\ref{elastic_tensor}) and (\ref{relaxation_tensor}), respectively, we use an opti\-mi\-za\-tion-based method described in detail in \cite{JOB}. 
We approximate the integral term by the right rectangle formula in each subinterval $[t_i, t_{i+1}]$ of $[0,T]$. We use the following approximation of the time integral operator in~\eqref{Mx8}:
\begin{equation*}
    \int_0^{t_n}{\cal B}(t_n-s)\bvarepsilon(\bu(s))\,ds \approx k \sum_{j=1}^{n} {\cal B}(t_n-t_j)\bvarepsilon(\bu_j).    \label{eq:integral}
\end{equation*}

\noindent
Therefore, using \eqref{MS} we have
\begin{align}
    ((\cS\bu)_i, \bv)_V &= \left(k \sum_{j=1}^{i} {\cal B}(t_i-t_j)\bvarepsilon(\bu_j), \, \bvarepsilon(\bv)\right)_Q \label{eq:Sh}  \\
    &= ((\cS \bu)_{i-1}, \bv)_V + (k {\cal B}(0) \bu_i, \bv)_V = ((\cS \bu)_{i-1}, \bv)_V + (k b \bu_i, \bv)_V\nonumber
\end{align}
for all $\bu, \bv \in V^h$ and $i = 1, \dots, n$.
Note that in $i$-th time step the  values of $\bu_0, \dots , \bu_{i-1}$ are known and, therefore, $(\cS\bu)_{i-1}$ is known too.
Thus, for every time step $i$, we use \eqref{eq:Sh} in order to introduce the cost functional $\mathcal{L}_i\colon V \to \R$ 
given by
\begin{equation*}
\mathcal{L}_i(\bw^h)=\frac{1}{2} (A \bw^h + k b \bw^h, \bw^h)_V
+j(\bw^h) +((\cS\bu^{hk})_{i-1} - \fb_i \,,\bw^h)_V\label{L}
\end{equation*}
for all $\bw^h \in V^h$, which is a convex functional. 
We are now in a position to find a sequence of minimizers of functionals $\mathcal{L}_i$, i.e., solve the following optimization problem. 

\medskip
\noindent
\textbf{Problem ${\cP_{1}^{Oh}}$.} {\it Find $\bu^{hk} = \{\bu_i^{hk}\}_{i=1}^N \subset V^h$ such that }
\begin{align*}
	0 \in & \partial \mathcal{L}_i( \bu_i^{hk} )\qquad  \forall\ i=1,\dots,N.
\end{align*}

\medskip

It can be shown that the operators appearing in inequality \eqref{P1_disc} satisfy the assumptions considered in [8] and, furthermore, the optimization Problem ${\cP_{1}^{Oh}}$ is equivalent to Problem $\cP_1^{Vh}$. For a deeper insight into the theory related to the optimization approach and the appropriate potential energy functional we mention the book \cite{W}. Moreover,
a more detailed analysis of the selected approach compared to other widely used methods can be found in \cite{OJB}.
To solve Problem ${\cP_{1}^{Oh}}$ we use our original software \textit{Conmech} that is a user-friendly tool written entirely in Python for conducting contact simulations and analyzing results.
To enhance the performance of native Python, we utilized the just-in-time compiler \textit{Numba} \cite{NUMBA}.
The package provides comprehensive support for simulations, ranging from easy definition of the shape and material properties of the body to generating computational meshes and performing empirical error analysis.
It supports simulations for both static, quasistatic and dynamic problems, in two or three dimensions.
The goal of the package is to easily extend existing models with additional physical effects, which is achieved thanks to the modularity of the software.
The package is open-source and provided under the GPL-3.0 license.

\begin{figure}[h!]
\centering
    \includegraphics[width=\linewidth]{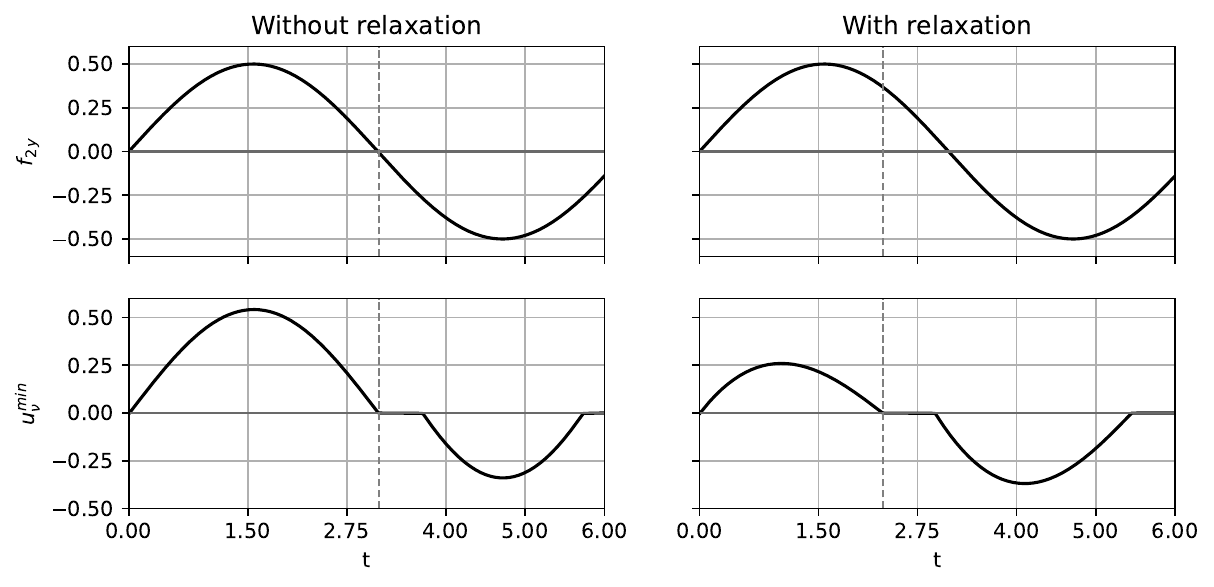}
 \caption{Time evolution of tractions (top row) and normal displacements (bottom row), in the case
 without relaxation (right column) and with relaxation (left column).} \label{img::fp}
\end{figure}

\begin{figure}
	\centering
	\hspace*{-1cm}
	\includegraphics[width=\linewidth]{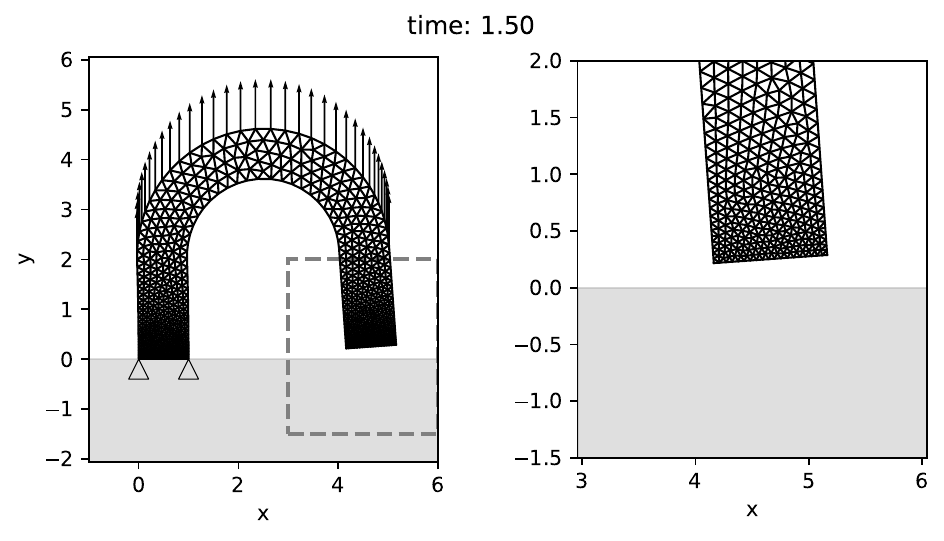}
	\hspace*{-1cm}
	\includegraphics[width=\linewidth]{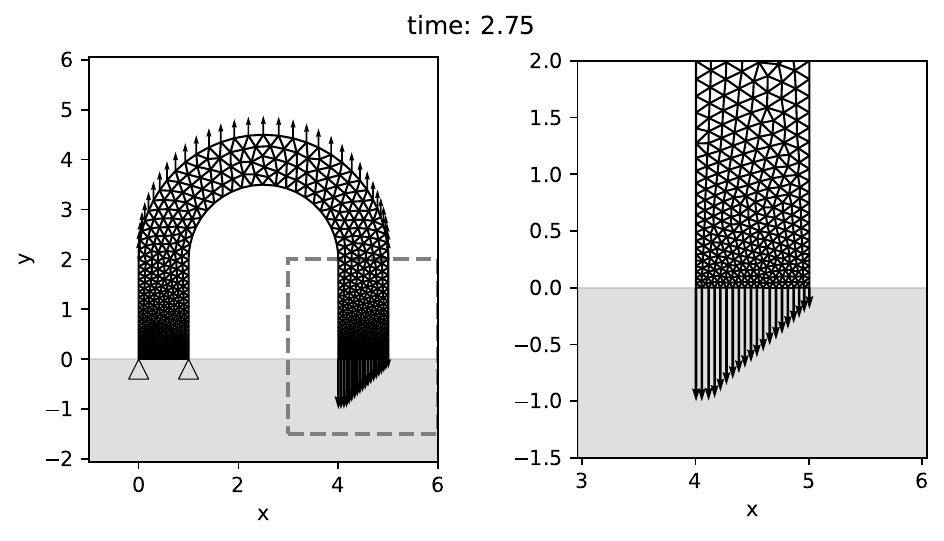}
	\caption{Deformed configuration of the body and stress vectors at $t=1.5\, \unit{s}$ and $t=2.75\, \unit{s}$.} \label{img::comp1}
\end{figure}

\begin{figure}
	\centering
	\hspace*{-1cm}
	\includegraphics[width=\linewidth]{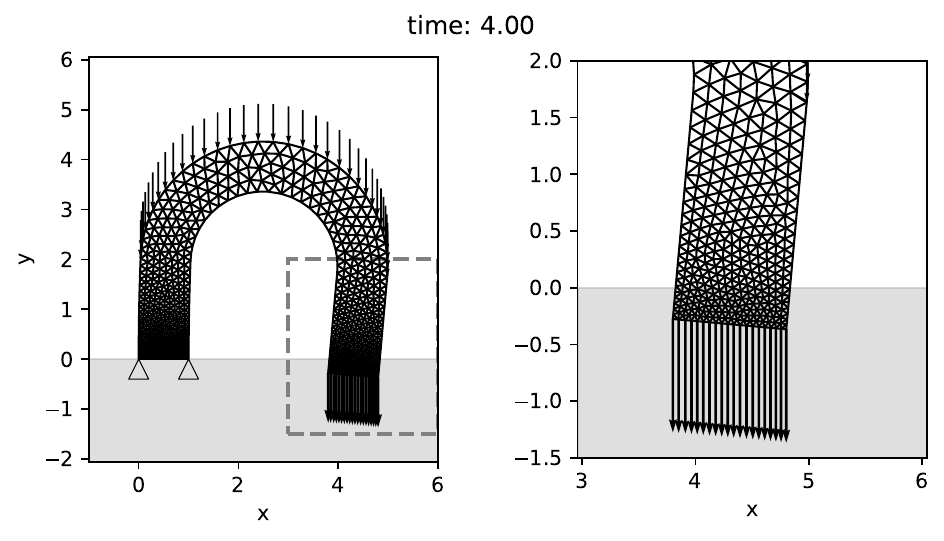}
	\hspace*{-1cm}
	\includegraphics[width=\linewidth]{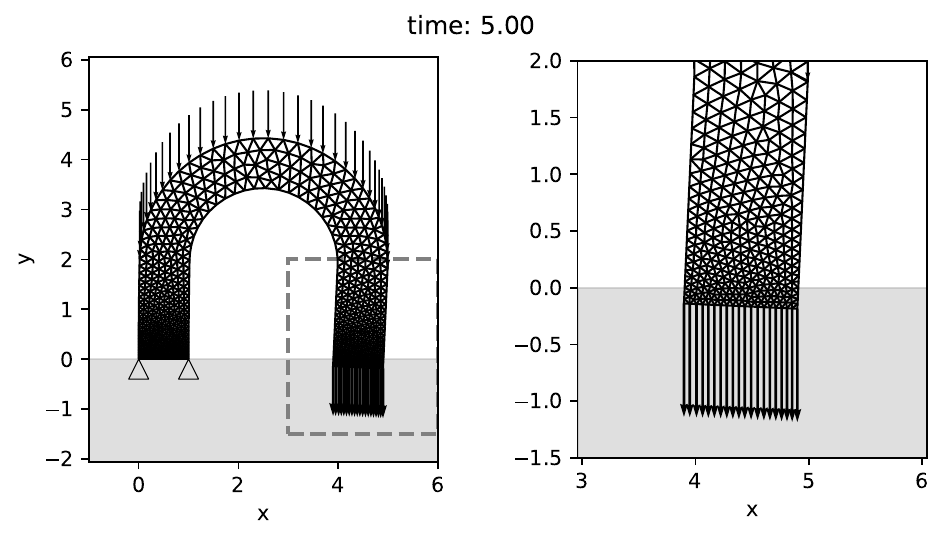}
	\caption{Deformed configuration of the body and stress vectors at $t=4\, \unit{s}$ and $t=5\, \unit{s}$.} \label{img::comp2}
\end{figure}

Our numerical results are presented in Figures \ref{img::fp}--\ref{img::comp2} and are described in what follows.

First, in the upper row  of Figure~\ref{img::fp} we plot the graph of the function $f_{2y}(t)=10\, \sin t$ scaled by a factor $0.2$. In the lower row we plot the evolution in time of minimum of the normal displacement  on the  potential contact surface $\Gamma_3$.  We consider two cases: the case when the body has a purely elastic behavior (i.e., the relaxation coefficient $b$ vanishes, see the left column of the figure)  and the case when the body  has a viscoelastic behavior. In this case the relaxation is taken into account (i.e., $b = 10^4 \, \unit{N m^{-2} s^{-1}}$, see the  right column of the figure).
It results from this figure that, as expected, the relaxation reduces the sensitivity of the body to  changes in applied forces and constrains it to return to its reference configuration faster.
This is particularly visible at the time moment marked by the dashed line, which represents the moment when the contact between the body and the foundation arises along the entire boundary $\Gamma_3$. This time moment we denote by $t_c$. 
In the elastic case (i.e., without relaxation) we have $t_c>2.75 \, \unit{s}$ and we note that $t_c$ coincides with the moment when the forces acting on the body vanish. However, in the  viscoelastic case (i.e., with relaxation), the  normal displacements on the contact surface are smaller and $t_c<2.75\, \unit{s}$. This behavior shows that, in this case, the body comes back to its reference configuration earlier.

Figures~\ref{img::comp1} and \ref{img::comp2} represent the current configuration of the viscoelastic body at various time moments, together with the external tractions and the opposite of the stress vectors on the potential contact surface. There, for
clarity, the length of the vectors representing external forces has been scaled by a factor of $0.2$, and the length of the arrows representing the opposite of the stress vector by a factor of $0.1$.
Moreover, on the left side, a zoom-in view of the contact boundary surface is presented. Note that in these two figures the stress vectors on the contact surface are vertical and, therefore, they reduce to their normal component. This result is in agreement with the assumption that the contact is frictionless.

The first row of Figure~\ref{img::comp1} concerns the time moment $t=1.5\, \unit{s}$. At that moment the applied
traction is upward and there is separation between the body and the foundation. As a consequence, the stress vector on the boundary $\Gamma_3$ vanishes. The second row of the figure  concerns the time moment $t=2.75\, \unit{s}$. At  that moment the contact arises, even if the applied traction is directed upwards. The explanation arises from the relaxation term in the constitutive law
in which the coefficient $b$ is large enough and, as a consequence, the resulted  stress push the body towards the foundation.
Note that at $t=2.75\, \unit{s}$  the contact is without penetration, since the absolute value of the normal stress is below the value of the yield limit $F=10$.

Figure~\ref{img::comp2}  concerns the time moments $t=4\, \unit{s}$ and $5\, \unit{s}$. At these moments the applied forces are downward, the contact arises on all the points of $\Gamma_3$ and is with penetration.  At these moments the magnitude of the normal stress reaches the yield limit $F$ and remain constant regardless of the depth of penetration.

\vskip 2mm
\noindent {\bf Acknowledgments}\\
The project has received funding from the European Union's Horizon 2020 Research and Innovation Programme under the Marie Sklo\-do\-wska-Curie grant agreement no.\ 823731 CONMECH.
The first two authors are supported by
National Science Center, Poland, under project OPUS no. 2021/41/B/ST1/01636.
The first author of the publication received an incentive scholarship from the funds of the program Excellence Initiative - Research University at the Jagiellonian University in Krakow.

\end{document}